\documentclass[12pt]{amsart}
\usepackage[all]{xy}
\usepackage{amssymb}
\usepackage{amsthm}
\usepackage{hyperref}
\usepackage{todonotes}
\hypersetup{colorlinks=true,linkcolor=blue,citecolor=magenta}
\usepackage{amsmath}
\usepackage{mathtools}
\usepackage{amscd,enumitem}
\usepackage{verbatim}
\usepackage{eurosym}
\usepackage{float}
\usepackage{color}
\usepackage{cases}
\usepackage{bm}
\usepackage{dcolumn}
\usepackage[mathscr]{eucal}
\usepackage[all]{xy}
\usepackage{hyperref}
\usepackage{bbm}
\usepackage{comment}
\usepackage{adjustbox}
\usepackage[textheight=8.77in, textwidth=6.8in]{geometry}

\usepackage[]{breqn}

\newtheorem*{thm*}{Theorem}
\newtheorem*{conj*}{Conjecture}

\newtheorem*{remark}{Remark}

\newtheorem{theorem}{Theorem}[section]

\newtheorem{lemma}[theorem]{Lemma}

\newtheorem{proposition}[theorem]{Proposition}

\newtheorem*{parityconjecture}{Parity Conjecture}

\newtheorem*{TwoRemarks}{Two Remarks}

\newtheorem{definition}[theorem]{Definition}

\newtheorem{corollary}[theorem]{Corollary}

\newcommand{\CL}{\mathrm{CL}}
\newcommand{\Z}{\mathbb{Z}}
\newcommand{\Q}{\mathbb{Q}}

\newcommand{\R}{\mathbb{R}}
\newcommand{\SL}{\operatorname{SL}}

\newcommand{\divides}{\mid}

\makeatletter
\DeclarePairedDelimiterX{\pmodx}[1]{(}{)}{{\operator@font mod}\mkern6mu#1}
\renewcommand{\pmod}{%
  \allowbreak
  \if@display\mkern18mu\else\mkern8mu\fi
  \pmodx
}
\makeatother

\numberwithin{equation}{section}

\begin{document}
\title{On class numbers, torsion subgroups, and quadratic twists of elliptic curves}

\author[T. Blum, C. Choi, A. Hoey, J. Iskander, K. Lakein, and T. Martinez]{Talia Blum, Caroline Choi, Alexandra Hoey, Jonas Iskander, Kaya Lakein, and Thomas C. Martinez }
\address{Department of Mathematics, Massachusetts Institute of Technology, Cambridge, MA 02139}
\email[T. Blum]{taliab@mit.edu}
\email[A. Hoey]{ahoey@mit.edu}
\address{Department of Mathematics, Stanford University, Stanford, CA 94305}
\email[C. Choi]{cchoi1@stanford.edu}
\email[K. Lakein]{epi2@stanford.edu}
\address{Department of Mathematics, Harvard University, Cambridge, MA 02138}
\email[J. Iskander]{jonasiskander@college.harvard.edu}
\address{Department of Mathematics, Harvey Mudd College, Claremont, CA 91711}
\email[T. Martinez]{tmartinez@hmc.edu}

%

\begin{abstract}
The Mordell-Weil groups $E(\Q)$ of elliptic curves influence the structures of their quadratic twists $E_{-D}(\Q)$ and the ideal class groups $\CL(-D)$ of imaginary quadratic fields. For appropriate $(u,v) \in \Z^2$, we define a family of homomorphisms $\Phi_{u,v}: E(\mathbb{Q}) \rightarrow \CL(-D)$ for particular negative fundamental discriminants $-D:=-D_E(u,v)$, which we use to simultaneously address questions related to lower bounds for class numbers, the structures of class groups, and ranks of quadratic twists. Specifically, given an elliptic curve $E$ of rank $r$, let $\Psi_E$ be the set of suitable fundamental discriminants $-D<0$ satisfying the following three conditions: the quadratic twist $E_{-D}$ has rank at least 1; $E_{\text{tor}}(\Q)$ is a subgroup of $\mathrm{CL}(-D)$; and $h(-D)$ satisfies an effective lower bound which grows asymptotically like $c(E) \log (D)^{\frac{r}{2}}$ as $D \to \infty$. Then for any $\varepsilon > 0$, we show that as $X \to \infty$, we have
    $$\#\, \left\{-X < -D < 0: -D \in \Psi_E  \right \} \, \gg_{\varepsilon} X^{\frac{1}{2}-\varepsilon}.$$
In particular, if $\ell \in \{3,5,7\}$ and $\ell \mid |E_{\mathrm{tor}}(\Q)|$, then the number of such discriminants $-D$ for which $\ell \divides h(-D)$ is $\gg_{\varepsilon} X^{\frac{1}{2}-\varepsilon}.$
Moreover, assuming the Parity Conjecture, our results hold with the additional condition that the quadratic twist $E_{-D}$ has rank at least 2.
\end{abstract}

\maketitle

\vspace{-0.5cm}
\section{Introduction and statement of results}\label{sec1}

Ideal class groups $\CL(-D)$ of imaginary quadratic fields $\mathbb{Q}(\sqrt{-D})$ are finite abelian groups isomorphic to the groups of positive definite integral binary quadratic forms of discriminant $-D <0$ studied by Gauss. 
Although Gauss conjectured that the class number $h(-D)$ tends to infinity as $D \rightarrow \infty,$ no lower bound on class numbers was established until the 1930s, when Siegel \cite{Siegel} proved that for any $\varepsilon > 0$, there exist constants $c_1(\varepsilon), c_2(\varepsilon) > 0$ for which 
    $$c_1(\varepsilon)\,D^{\frac{1}{2}-\varepsilon} \leq h(-D) \leq c_2(\varepsilon)\,D^{\frac{1}{2}+\varepsilon}.$$
However, because the constant $c_1(\varepsilon)$ depends on the truth or falsity of the Generalized Riemann Hypothesis, Siegel's lower bound is not effective. In the 1980s, Goldfeld \cite{Goldfeld1, Goldfeld2}, Gross and Zagier \cite{GrossZagier}, and Oesterl\'e \cite{Oesterle} used deep results on the Birch and Swinnerton-Dyer Conjecture to prove the effective lower bound
    \begin{equation}\label{GrossZagier-bound}
        h(-D) > \frac{1}{7000}\,\log(D) \prod_{\substack{p\,|\,D\text{ prime} \\
        p\,\neq\, D}} \left(1-\frac{\lfloor 2\sqrt{p}\rfloor}{p+1}\right).
    \end{equation}
    
Recent work improves on this bound by exploiting \textit{ideal class pairings} $E(\Q) \times E_{-D}(\Q) \to \CL(-D),$ first studied by Buell, Call, and Soleng \cite{Buell, BuellCall, Soleng}. Griffin, Ono, and Tsai \cite{G-O, G-O-T} obtain an effective lower bound of the form $h(-D) \geq c_1(E) \log(D)^{\frac{r}{2}}$ for certain families of discriminants, which improves on~\eqref{GrossZagier-bound} when the rational Mordell-Weil rank of the elliptic curve $r := r_\Q(E) \geq 3$.

A famous conjecture of Goldfeld asserts that for a given elliptic curve, asymptotically half of all quadratic twists have rank 0 (resp.\@ 1), which raises the question of how many quadratic twists have rank at least $2$  \cite{GoldfeldConjecture}. For any elliptic curve $E$ with $j(E) \neq 0, 1728,$ Stewart and Top \cite[Theorem~3]{StewartTop} unconditionally prove a lower bound of the form
    $$\# \{ -X<-D<0 \ : \ r_{\Q}(E_{-D})\geq 2\} \geq c_2(E) \cdot \frac{X^{\frac{1}{7}}}{\log(X)^2},$$
where $c_2(E)$ is a constant depending on the elliptic curve. However, improved lower bounds can be obtained by assuming the Parity Conjecture.

\begin{parityconjecture}
Let $E/\Q$ be an elliptic curve with Hasse-Weil $L$-function $L(E,s)$, and let \(\epsilon \in \{\pm 1\}\) be the sign of its functional equation. Then \(r_\Q(E) \equiv 0 \pmod{2}\) if and only if \(\epsilon = 1\).
\end{parityconjecture}

\noindent Gouv\^ea and Mazur \cite[Theorem~2]{GouveaMazur} assume the Parity Conjecture and obtain the lower bound
    $$\# \{ -X<-D<0 \ : \ r_{\Q}(E_{-D})\geq 2\} \gg_{\varepsilon} X^{\frac{1}{2}-\varepsilon}.$$
Moreover, assuming the Parity Conjecture, Griffin, Ono, and Tsai prove for a particular family of elliptic curves that the number of fundamental discriminants $-X < -D < 0$ for which their lower bound for $h(-D)$ applies and the rank of $E_{-D}$ is at least $2$ is asymptotically greater than $X^{\frac12-\varepsilon}$ \cite[Theorem~1.2]{G-O-T}.

In addition to using ideal class pairings to obtain lower bounds for class numbers and ranks of quadratic twists, it is natural to ask whether they can also be used to study the structures of class groups. In the 1980s, Cohen and Lenstra \cite{C-L} conjectured that for any odd\footnote{
For fundamental discriminants $-D < 0$, Gauss's genus theory shows that $h(-D)$ is odd if and only if we have $D = -4, D= -8$, or $D = -p$ for some prime $p \equiv 3$ mod $4$.} prime $\ell$,
    $$\lim_{X\,\to\,\infty} \frac{\#\{-X < -D < 0 :  \ell \divides h(-D),\; -D\text{ fundamental}\}}
    {\#\{-X < -D < 0 :-D\text{ fundamental}\}} = 1- \prod_{k\,=\,1}^\infty \left(1 - \frac{1}{\ell^k}\right).$$
However, little is known about the truth of their conjecture. Davenport and Heilbronn \cite{DavenportHeilbronn} proved a lower bound on the density of the class numbers $h(-D)$ which are not divisible by 3, and Kohnen and Ono \cite{KohnenOno} proved a lower bound for the number of $h(-D)$ not divisible by any odd prime. The current best lower bound for the number of $h(-D)$ divisible by an odd prime $\ell$ is due to Soundararajan \cite{Soundararajan}:
$$\#\{-X < -D < 0 :\ell \divides h(-D),\;-D\text{ fundamental}\} \gg X^{\frac{1}{2}+\frac{1}{\ell}}.$$
Note that these lower bounds fall short of a positive proportion of the negative fundamental discriminants, and that $X^{\frac{1}{2}}$ is considered the current standard.

In this paper, we prove a result that simultaneously addresses lower bounds for class numbers, the structures of class groups, and ranks of quadratic twists.
We consider elliptic curves of the form $E: y^2 = x^3 + a_4x + a_6$ with $a_4, a_6 \in \Z$, and their quadratic twists given by the non-standard model
$$E_{-D}: -\frac{D}{4}\cdot y^2 = x^{3} +a_4x+a_6.$$
To prove our result, we study a family of maps
    $$\Phi_{u,v}\colon \,E(\Q)\to \CL(-D_E(u,v))$$
defined in Section~\ref{sec2}, where $-D_E(u,v)$ is the family of negative discriminants defined by $$-D_E(u,v) = -4d_E(u,v) := -4v\,(u^3+a_4\,uv^2-a_6\,v^3), \qquad u,v \in \Z^+.$$ Here, $D_E(u,v)$ and $\Phi_{u,v}$ are defined in terms of the coefficients $a_4$ and $a_6$, and hence depend on the particular model for the elliptic curve.

We define a notion of \textit{map-suitable pairs} $(u,v) \in \Z^2$ in Section~\ref{sec2}, and we prove that for all such pairs $(u,v)$, the map $\Phi_{u,v}$ is well-defined with the following property.

\begin{theorem}\label{Theorem : Momomorphism}
If the pair $(u,v) \in \Z^2$ is map-suitable for $E: y^2 = x^3+a_4x+a_6$, then the map $\Phi_{u, v}:E(\Q) \rightarrow \CL(-D_E(u,v))$ is a homomorphism.
\end{theorem}

To state our applications of Theorem~\ref{Theorem : Momomorphism}, we first fix some notation. Recall that for $x = \frac{m}{n} \in \Q,$ where $\gcd(m,n)=1$, the \textit{na\"ive height} of $x$ is defined by $H(x) := \max(|m|, |n|),$ and the \textit{Weil height} is $h_W(x) := \log(H(x)).$ If $P = (x,y) = \left(\frac{A}{C^2}, \frac{B}{C^3}\right) \in E(\mathbb{Q})$
with $\gcd(A,C) = \gcd(B,C) = 1$, the \textit{na\"ive height} of $P$ is defined by $H(P) := H(x)$, its \textit{Weil height} is $h_W(P) := h_W(x)$, and its \textit{canonical height} is given by
$$\hat{h}(P) := \frac{1}{2} \cdot \lim_{n\rightarrow\infty} \frac{h_W(nP)}{n^2}.$$
We denote the $j$-invariant and discriminant of an elliptic curve $E$ by $j(E)$ and $\Delta(E)$, respectively. Let $\Omega_r:=\pi^{\frac{r}{2}}/\Gamma\left (\frac{r}{2}+1\right)$ denote the volume of the $\R^r$-unit ball, and let $\mathcal{P} = \{P_1, \dots, P_r\}$ be a set of $r$ linearly independent points of infinite order in $E(\Q)$. We define the \textit{regulator} and the \textit{diameter} of $\mathcal{P}$ by \begin{equation}\label{eq: regulator and diameter}
    R_\Q(\mathcal{P}) := \det(\langle P_i, P_j \rangle)_{1\, \leq \,i\,,\,j\, \leq\, r} \quad \text{and} \quad d(\mathcal{P}) := \max_{\delta_i \in \{0\,,\,\pm 1\}} 2\hat{h}\left(\sum_{i=1}^r\delta_i P_i\right),
\end{equation} respectively, where $\langle P_i, P_j \rangle := \frac12 \left(\hat{h}(P_i+P_j)-\hat{h}(P_i)-\hat{h}(P_j)\right)$ denotes the N\'eron-Tate height pairing. For notational convenience, we also define the constants
\begin{equation}\label{eq: cE}
c_G(\mathcal{P}):=\frac{|G|}{\sqrt{R_{\Q}(\mathcal{P})}}\cdot \Omega_r \quad \text{and} \quad  \delta(E) := \frac{1}{8}\,h_W(j(E))+\frac{1}{12}\,h_W(\Delta(E))+\frac{5}{3}
\end{equation} 
for a subgroup $G$ of $E_{\mathrm{tor}}(\Q)$.  For $0 < \varepsilon < \frac{1}{2}$, let
\begin{equation}\label{T}
 T_{E}(u,v, \varepsilon) := \frac{1}{8}\,\log \left(\frac{4\,d_E(u,v)^{1-\varepsilon}}{uv+v^2}\right) - \frac{\delta(E)}{4}.
\end{equation}
Using the homomorphism $\Phi_{u,v}$ defined in Section \ref{sec2}, we obtain the following result.

\begin{theorem}\label{BigTheorem}
Suppose $\mathcal{P}$ is a set of linearly independent points in $E(\Q)$ and $G$ is a subgroup of $E_{\mathrm{tor}}(\Q)$. Let $0 < \varepsilon_1 < \frac 12$ and $0 \leq \alpha < \frac{1}{2} - \varepsilon_1$, and let $\Psi_E$ denote the set of fundamental discriminants $-D_E(u,v)<0$ with $ u, v > 0$ such that the following are true.
\begin{enumerate}
\item The point $\left(-\frac{u}{v},\frac{1}{v^{2}}\right) \in E_{-D_E(u,v)}(\Q)$ has infinite order.
\item We have that $h(-D_E(u,v)) \geq c_G(\mathcal{P})\cdot \big(T_E(u,v, \varepsilon_1)^{\frac{r}{2}} - r\sqrt{d(\mathcal{P})} \cdot T_E(u,v, \varepsilon_1)^{\frac{r-1}{2}}\big)$, where $T_E(u, v, \varepsilon_1) > \frac{\alpha}{8}\log(d_E(u, v)) + \frac{d(\mathcal{P})}{4}$.
\item The class group $\CL(-D_E(u,v))$ contains a subgroup isomorphic to $E_{\mathrm{tor}}(\Q)$.
\end{enumerate}
Then for any $\varepsilon_2 > 0$, we have
\begin{equation}\label{bigboy}
 \#\, \left\{-X < -D < 0: -D \in \Psi_E \right\} \, \gg_{\varepsilon_2} X^{\frac{1}{2}-\varepsilon_2}.
\end{equation}
Moreover, if the conductor \(N(E)\) is not a perfect square and $4N(E)$ divides $a_4$ and $a_6$, then assuming the Parity Conjecture, we may also require that $r_{\Q}(E_{-D_E(u,v)}) \geq 2$.
\end{theorem}

\begin{TwoRemarks}  \ \

\noindent (1)  The condition that $4N(E) \mid a_4,a_6$ can be guaranteed by choosing an appropriate model for \(E\). 

\noindent (2) The result in the abstract follows from Theorem~\ref{BigTheorem} by letting $\mathcal{P} \subseteq E(\Q)$ be a set of $r$ linearly independent points such that $E(\Q) = \langle \mathcal{P} \rangle \oplus E_{\mathrm{tor}}(\Q)$, setting $G := E_{\mathrm{tor}}(\Q)$, and choosing $\alpha > 0$.
\end{TwoRemarks}

To demonstrate that it is easy to find examples for which the bound given in Theorem~\ref{BigTheorem} improves on \eqref{GrossZagier-bound}, we apply Theorem~\ref{BigTheorem} to several explicit infinite families of elliptic curves with high rank and specified torsion subgroup. Recall that Mazur's Theorem \cite{Mazur} completely classifies all torsion subgroups of elliptic curves over $\Q$: 
\[
    E_{\text{tor}}(\Q) = \begin{cases}
    \Z/k\Z, &\text{for } k = 2,\dots, 10, \text{ and } 12, \\
    \Z/2\Z \times \Z/2m\Z, &\text{for } m=1,\dots,4. 
    \end{cases}
\] 
Infinite families of elliptic curves with positive rank are known for all torsion subgroups except $\Z/9\Z$, $\Z/10\Z$, $\Z/12\Z$, and $\Z/2\Z \times \Z/8\Z$ \cite{Dujella}. Using the infinite families described in \cite{Dujella, DujellaPeral, DujellaPeral2, Elkies,pElkies, Kihara22, Kihara4, Kulesz}, we obtain explicit lower bounds for $h(-D)$ that often improve on \eqref{GrossZagier-bound}. 

\begin{theorem}\label{TheoremInfiniteFamilies}
Let $\mathcal{S} := \{\Z/n\Z : 2 \leq n \leq 8\} \cup \{\Z/2\Z \times \Z/2n\Z : 1\leq n \leq 3\}$. Then for each group $G \in \mathcal{S}$, there exists an infinite family\footnote{We give a link to explicit formulas for these families of elliptic curves in Appendix~\ref{Appendix}.} $\mathcal{E}_G$ of elliptic curves such that the following conditions hold: \begin{enumerate}
    \item There exist integral polynomials $a_4(t_1, \dots, t_k)$ and $a_6(t_1, \dots, t_k)$, with $k=2$ if $G = \Z/2\Z$, $k = 3$ if $G = \Z/3\Z$, and $k = 1$ otherwise, such that each element of $\mathcal{E}_G$ is modeled by the curve $y^{2} = x^{3} + a_4(t_1, \dots, t_k)x + a_6(t_1, \dots, t_k)$ for some $t_1, \dots, t_k \in \Z$.
    \item For all but finitely many $E \in \mathcal{E}_G$, there exists a set $\mathcal{P}$ of $r_{\mathrm{min}}(\mathcal{E}_G)$ linearly independent points in $E(\Q)$ whose coordinates are rational functions in $t_1, \dots, t_k$.
    \item For each $E \in \mathcal{E}_G$, if we have $E:\; y^{2} = x^{3} + a_4(t_1, \dots, t_k)x + a_6(t_1, \dots, t_k)$, then there exist positive constants $m(\mathcal{E}_G)$ and $\mu(\mathcal{E}_G)$ such that the value of $c_G(\mathcal{P})$ is greater than the value $c_{\mathrm{min}}(\mathcal{P})$ given in the following table,\footnote{The family of elliptic curves with torsion subgroup $\Z/3\Z$ has $3$ parameters, two of which were fixed ($t_1=2,\, t_2=4,\, t_3=6$) to compute the values listed in Table~\ref{tabel: torsion subgroup}. The subscript indicates the parameter that was varied.} where $T := \log(\max_i|t_i| + m(\mathcal{E}_G)) + \mu(\mathcal{E}_G)$.
\end{enumerate}

\begin{table}[h!]
\begin{adjustbox}{width=0.7\linewidth,center}
\begin{tabular}{c|c|c|c|c}
 Torsion Subgroup & $r_{\mathrm{min}}(\mathcal{E}_G)$ & $c_{\mathrm{min}}(\mathcal{P})$ & $m(\mathcal{E}_G)$ & $\mu(\mathcal{E}_G)$ \\ \hline 
$\Z/2\Z$ & $8$ &   $4.050\times10^{-10}\cdot T^{-4}$ & $17$ & $0.26$ \\
\;$\Z/3\Z_1$ & $6$ & $2.320\times10^{-6}\cdot T^{-3}$ & $2$ & $1.45$ \\
\;$\Z/3\Z_2$ & $6$ & $2.320\times10^{-6}\cdot T^{-3}$ & $2$ & $1.66$ \\
\;$\Z/3\Z_3$ & $6$ & $2.320\times10^{-6}\cdot T^{-3}$ & $2$ & $0.85$ \\
$\Z/4\Z$ & $5$ & $1.693\times10^{-6}\cdot T^{-\frac{5}2}$  &  $2$ & $0.18$ \\
$\Z/5\Z$ & $2$ & $ 6.732\times10^{-1}\cdot T^{-1}$  &  $2$ & $0.56$  \\
$\Z/6\Z$ & $2$ & $7.968 \times 10^{-1} \cdot T^{-1}$  & $2$ & $0.62$ \\
$\Z/7\Z$ & $1$ & $3.315 \times 10^{0} \cdot T^{-\frac{1}2}$  &  $3$ & $0.33$  \\
$\Z/8\Z$ & $2$ & $2.990 \times 10^{-1} \cdot T^{-1}$   & $6$ & $0.25$ \\
$\Z/2\Z \times \Z/2\Z$ & $6$ & $1.808\times 10^{-7} \cdot T^{-3}$ & $1$ & $0.20$ \\
$\Z/2\Z\times \Z/4\Z$ & $4$ & $1.016 \times 10^{-2} \cdot T^{-2}$  & $1$ & $0.86$\\
$\Z/2\Z\times \Z/6\Z$ & $2$ & $9.058 \times 10^{-1} \cdot T^{-1}$ & $3$ & $0.24$
\end{tabular}
\end{adjustbox}
\vspace{0.2cm}
\caption{\label{tabel: torsion subgroup} Lower bounds for $c_G(\mathcal{P})$ where $E$ has a specified torsion subgroup.}
\vspace{-6mm}
\end{table}
\end{theorem}

In practice, the discriminant must be quite large for the lower bound for $h(-D)$ given in Theorem \ref{BigTheorem} to improve on \eqref{GrossZagier-bound} for these families, but once discriminants exceed a certain threshold, the improvement is substantial. 
For instance, fix $\varepsilon = 1/1000$, and consider the elliptic curve $E_6$ obtained from the family $\mathcal{E}_{\Z/6\Z}$ described in Theorem \ref{TheoremInfiniteFamilies} by setting $t_1 = 0$. One can verify computationally\footnote{A link to the code used to compute these results is contained in Appendix \ref{Appendix}.} that for suitable $(u,v)$ given approximately by $(2728 \times 2^{2800}, 2^{2800})$, $(2728 \times 2^{3000}, 2^{3000})$, and $(2728 \times 2^{30000}, 2^{30000})$, which yield
discriminants $D_{E_6}(u, v)$ on the order of $e^{7778}$, $e^{8332}$, and $e^{83192}$, respectively, we obtain lower bounds for $h(-D)$ of $5$, $16$, and $2324$; \eqref{GrossZagier-bound} meanwhile gives lower bounds of approximately 2, 2, and 12. If one is willing to consider much larger discriminants, then the families of higher rank from Theorem \ref{TheoremInfiniteFamilies} yield better bounds. For example, for the member $E_2$ of the family $\mathcal{E}_{\Z/2\Z}$ determined by $t_1 = 0$ and $t_2 = 1$, fixing $\varepsilon = 1/1000$ and choosing $(u, v) \approx (290291 \times 2^{319618}, 2^{319618})$ suitable, we find that $D_E(u, v) \approx e^{886195}$, and that our lower bound for the class number is approximately $1.3 \times 10^7$, whereas \eqref{GrossZagier-bound} gives a lower bound of approximately 127. 


This paper is organized as follows. In Section \ref{sec2}, we define the family of maps $$\Phi_{u,v}: E(\Q) \rightarrow \CL(-D_E(u,v))$$ and prove Theorem~\ref{Theorem : Momomorphism}. Using this result, in Section~\ref{sec3} we give an explicit lower bound for $h(-D_E(u,v)).$ Finally, in Section~\ref{sec4}, we use a theorem of Gouv\^ea and Mazur \cite{GouveaMazur} to count the number of discriminants in our family $-D_E(u,v)$ that fulfill the criteria given in Theorem~\ref{BigTheorem}. A discussion of Theorem~\ref{TheoremInfiniteFamilies} is given in Appendix~\ref{Appendix}.
\vspace{0.5cm}

\section*{Acknowledgements}
The authors would like to thank Professor Ken Ono for advising this project and for many helpful conversations and suggestions. We thank Professor N.\@ Elkies and W.-L.\@ Tsai for their valuable comments, as well as W.\@ Craig and B.\@ Pandey for their support, and the anonymous referees for their helpful comments. We give special thanks to Shengtong Zhang for his assistance in proving several key divisibility conditions. Finally, we thank the NSF (DMS-2002265), the NSA (H98230-20-1-0012), the Templeton World Charity Foundation, and the Thomas Jefferson Fund at the University of Virginia.
\vspace{0.5cm}

\section{A Homomorphism from Elliptic Curves to Class Groups}\label{sec2}

Consider an elliptic curve $E: y^2= x^3+a_4x+a_6$. Motivated by the ideal class pairings $E(\Q)\times E_{-D}(\Q) \rightarrow \CL(-D)$ studied by Buell, Call, Soleng, Griffin, and Ono \cite{Buell, BuellCall, G-O, Soleng},  we define a map $\Phi_{u, v}: E(\Q) \rightarrow \CL(-D_E(u, v))$ for the infinite family of negative fundamental discriminants $-D_E(u, v)$ with \textit{map-suitable} $(u,v)$:

\begin{definition} \label{Def:map-suitable}
We say that a pair $(u, v) \in \Z^2$ is \textbf{map-suitable for $\textbf{E}$} if the values $u$, $v$, $3u^2 + a_4v^2$, and $D_E(u, v)$ are positive.
\end{definition}

For notational convenience, we write $d_E(u,v):= \frac{D_E(u,v)}{4}$. To define our maps $\Phi_{u,v}$, we will need the following lemma.

\newpage
\begin{lemma} \label{Lemma:Phi}
Suppose $(u, v) \in \Z^2$ is map-suitable for $E$ and $-D_E(u,v)$ is a negative fundamental discriminant. Then
\begin{enumerate}
    \item $v$ is square-free and $\gcd(u,v) = 1.$
    \item If $P \in E(\Q)$ is not the point at infinity $\mathcal{O}$, there exist $A, B, C \in \Z$ with $\gcd(A, C) = \gcd(B, C) = 1$, $C > 0,$ such that $P = (\frac{A}{C^2}, \frac{B}{C^3}).$ Write $a = Av + C^2u$ and $g = \gcd(C, v).$ Then there exists an integer $\mu$ such that $C^3/g^2 \cdot \mu \equiv 1 \pmod {va/g^2}.$
    \item For any such $\mu,$ the following formula 
        \begin{equation} \label{PhiDefinition}
        \Phi_{u, v}(P, \mu) :=
            \frac{va}{g^2} \cdot x^2 + 2\mu\,\frac{Bv^2}{g^2} \cdot xy + \frac{\mu^2\,\frac{B^2v^4}{g^4}+d_E(u,v)}{\frac{va}{g^2}} \cdot y^2
        \end{equation}
    defines a positive definite binary quadratic form with discriminant $-D_E(u,v).$
\end{enumerate}
\end{lemma}

\begin{proof}
(1) follows from the assumption that $-D_E(u,v)$ is a fundamental discriminant. To prove (2), we suppose $p$ is a prime such that $p \divides \frac{C^3}{g^2}$ and $p \divides \frac{va}{g^2} = \frac{v(Av+C^2u)}{g^2}$, i.e., $g^2p \divides C^3$ and $g^2p \divides v(Av+C^2u)$. In particular, we see that $p \divides C$ and $p \divides v(Av + C^2 u)$, so $v(Av+C^2u) \equiv Av^2 \equiv 0 \pmod{p}$. Since $A$ and $C$ are coprime, we obtain $p \divides \gcd(C, v) = g$. Consequently, $p^3 \divides v(Av+C^2u)$, which implies that $p^2 \divides Av+C^2u$ because $v$ is square-free. This yields $Av \equiv 0 \pmod{p^2}$ and hence $p^2 \divides v$, which is impossible. Thus, $\gcd(\frac{C^3}{g^2}, \frac{va}{g^2}) = 1$. 

We now prove (3). A straightforward computation shows that the form in \eqref{PhiDefinition} has discriminant \(-D_E(u, v)\). To see that $\Phi_{u,v}(P, \mu)$ is positive definite for any $P \in E(\Q)$, we observe that map-suitability ensures that the coefficient of each power of $x$ in the polynomial expression $\big(-\frac{u}{v} - x\big)^3 + a_4\big(-\frac{u}{v} - x\big) + a_6$ is negative. Thus, $x^3 + a_4x + a_6 < 0$ for all $x \leq -\frac{u}{v}$, implying that $\frac{A}{C^2} > -\frac{u}v$, or equivalently, $a = C^2v(\frac{A}{C^2}+\frac{u}{v}) > 0$.

Finally, we show that the form in \eqref{PhiDefinition} has integral coefficients. The first two coefficients are clearly integers. For the third, note that we have $B^2v^4 = v(A^3v^3 + a_4AC^4v^3 + a_6C^6v^3)$ because $P \in E(\Q)$, and $C^6d_E(u,v) = v(C^6u^3+a_4C^6uv^2-a_6C^6v^3)$ by the definition of $d_E(u,v)$. This allows us to write \begin{equation*}
    B^2v^4 + C^6d_E(u,v) = v(A^3v^3 + C^6u^3 + a_4C^4v^2(Av + C^2u)) = va(A^2v^2 - AC^2uv + C^4u^2 + a_4C^4v^2),
\end{equation*} which in particular yields that $va \divides B^2v^4 + C^6d_E(u,v)$. From this, we see that 
\begin{equation*}
\mu^2\bigg(\frac{Bv^2}{g^2}\bigg)^2 + d_E(u,v) \equiv \bigg(\frac{C^3}{g^2}\bigg)^{-2}\bigg(\bigg(\frac{Bv^2}{g^2}\bigg)^2 + \frac{C^6d_E(u,v)}{g^4}\bigg) \equiv 0 \pmod[\Big]{\frac{va}{g^2}}.
\end{equation*}
\end{proof}

Using this lemma, we can define our map $\Phi_{u,v}:E(\Q) \rightarrow \CL(-D_E(u,v))$ for appropriate $(u,v)$.

\begin{theorem} \label{TheoremPhi}
Let the notation and hypotheses be as in Lemma~\ref{Lemma:Phi}. Then
\begin{enumerate}
    \item The class of $\Phi_{u,v}(P, \mu)$ in $\CL(-D_E(u,v))$ depends only on $P$ and not on the choice of $\mu,$ hence defines a map $\Phi_{u,v}: E(\Q) \rightarrow \CL(-D_E(u,v)),$ where by convention $\Phi_{u,v}(\mathcal{O})$ is set to be the identity of $\CL(-D_E(u,v)).$
    
    \item  If $\Phi_{u, v}(P)$ is the identity of $\CL(-D_E(u, v))$, then we have either $v \mid C$ or $\frac{va}{g^2} \geq d_E(u,v)$.
\end{enumerate}
\end{theorem}

\begin{remark}
Although the statement of Theorem~\ref{TheoremPhi} does not assign special significance to the parameters $u$ and $v$, in analogy to ideal class pairings, one may view them as specifying a point $Q = (-\frac{u}{v},\frac{1}{v^2})$ on the quadratic twist $E_{-D_E(u,v)}(\Q)$.
\end{remark}

\begin{proof}
(1) To see that \eqref{PhiDefinition} represents a single equivalence class under the action of $\SL_2(\Z)$, simply note that if $\mu_1$ and $\mu_2$ satisfy $\frac{C^3}{g^2}\mu_1 \equiv \frac{C^3}{g^2}\mu_2 \equiv 1 \pmod{\frac{va}{g^2}}$, then $\mu_1 \equiv \mu_2 \pmod{\frac{va}{g^2}}$, so the corresponding forms are properly equivalent under the $\SL_2(\Z)$ transformation $x \mapsto x + \frac{(\mu_2-\mu_1)Bv}{a}y, y \mapsto y$.

(2) If $P = \big(\frac{A}{C^2}, \frac{B}{C^3}\big)$ maps to the identity, then by comparing the first coefficient of the corresponding form to that of the identity form acted upon by a general element of \(\SL_2(\Z)\), one can verify that we must have $\frac{va}{g^2} = \alpha^2 + d_E(u,v)\gamma^2$ for some coprime $\alpha, \gamma \in \Z$. If $\gamma = 0$, then $\alpha^2 = 1$ and so $\frac{va}{g^2} = 1$, from which it easily follows that $v \divides C$. Otherwise, we easily obtain $\frac{va}{g^2} \geq d_E(u,v)\gamma^2 \geq d_E(u,v)$ as required.
\end{proof}

To prove Theorem~\ref{Theorem : Momomorphism}, we use Bhargava's formulation of the various composition laws for binary quadratic forms \cite[Section~2]{Bhargava}. For the remainder of this section, fix an elliptic curve $E: x^3+a_4x+a_6$ and a pair $(u,v) \in \Z^2$ that is map-suitable for $E$ such that $-D_E(u,v)$ is a negative fundamental discriminant.
Consider three finite points $P_i = \big(\frac{A_i}{C_i^2}, \frac{B_i}{C_i^3}\big) \in E(\Q)$, with $\gcd(A_i,C_i) = \gcd(B_i,C_i)$ = 1 and $C_i > 0$, that satisfy $P_1 + P_2 + P_3 = \mathcal{O}$. Let $l\,y = m\,x + n$ with $l, m, n \in \Z$ be the line whose intersections with $E$ are precisely the points $P_i$, counting multiplicities. For convenience, we will let $i$, $j$, and $k$ denote three distinct indices. Set $d := d_E(u, v)$, $C := C_1\,C_2\,C_3$, $M := C\cdot\frac{m}{l}$, $N := C\cdot\frac{n}{l}$, $b := Nv-Mu$, $a_i := A_i\,v + C_i^2\,u$, and $g_i := \gcd(C_i, v)$. Moreover, choose $\mu_i, \ell_i \in \Z$ such that $C_i^3\,\mu_i + v\,a_i\,\ell_i = g_i^2$, and set $q_i := \frac{b\,v\,\ell_i-M\,C_i\,\mu_i}{g_i}$. Lastly, define the values \begin{align*}
    \rho := v \cdot \frac{C}{g_1\,g_2\,g_3}, \quad \psi_i := v \cdot \frac{q_i}{g_j\,g_k}, \quad \phi_i := \frac{v\,q_j\,q_k - a_i\,g_j\,g_k}{C\,g_i}, 
    \end{align*}
    and
    \begin{align*}
    \theta := \frac{1}{C^2}\left(v\,q_1\,q_2\,q_3-a_1\,g_2\,g_3\,q_1-a_2\,g_1\,g_3\,q_2-a_3\,g_1\,g_2\,q_3+2\,b\,g_1\,g_2\,g_3\right).
\end{align*}
Then we have the following lemma.\footnote{The proof that $\phi_i, \theta$ are integral was communicated to us by Zhang.}

\begin{lemma} \label{TheoremBhargava} The Bhargava cube 
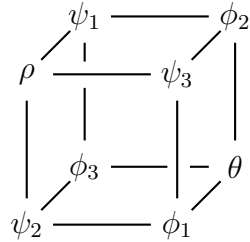
\begin{equation*}
\begin{tikzpicture}
\draw[thick] (0.35,0,0) -- (1.1,0,0);
\draw[thick] (1.4,0,0) -- (1.65,0,0);

\draw[thick] (2,0.35,0) -- (2,1.65,0);
\draw[thick] (1.65,2,0) -- (0.35,2,0);

\draw[thick] (0,1.65,0) -- (0,1.4,0);
\draw[thick] (0,1.05,0) -- (0,0.35,0);


\draw[thick] (0.35,0,2) -- (1.65,0,2);
\draw[thick] (2,0.35,2) -- (2,1.65,2);
\draw[thick] (1.65,2,2) -- (0.35,2,2);
\draw[thick] (0,1.65,2) -- (0,0.35,2);

\draw[thick] (0,0,.5) -- (0,0,1.5);
\draw[thick] (2,0,0.5) -- (2,0,1.5);
\draw[thick] (2,2,0.5) -- (2,2,1.5);
\draw[thick] (0,2,0.5) -- (0,2,1.5);

\draw(0,0,0) node{$\phi_3$};
\draw(0,2,0) node{$\psi_1$};
\draw(2,2,0) node{$\phi_2$};
\draw(2,0,0) node{$\theta$};

\draw(0,0,2) node{$\psi_2$};
\draw(0,2,2) node{$\rho$};
\draw(2,2,2) node{$\psi_3$};
\draw(2,0,2) node{$\phi_1$};
\end{tikzpicture}
\end{equation*} is integral of discriminant $-D_E(u, v)$ with associated primitive quadratic forms \begin{equation*}
    Q_i(x, y) := \frac{v\,a_i}{g_i^2} \cdot x^2 + 2\mu_i\,\frac{B_i\,v^2}{g_i^2} \cdot xy + \frac{\mu_i^2\,\frac{B_i^2\,v^4}{g_i^4}+d}{\frac{v\,a_i}{g_i^2}} \cdot y^2.
\end{equation*}
\end{lemma}

By direct computation, one verifies that the given cube has the desired discriminant and associated forms. The associated forms are primitive because the discriminant $-D_E(u,v)$ is assumed to be fundamental. It remains only to show that the cube is integral. To do so, we will use the following identities.

\begin{lemma} \label{lemma:observations}
We have the identities
\begin{align}
    a_1\,a_2\,a_3\,v &= b^2\,v^2 + C^2\,d \label{obs1}, \\
    A_1\,A_2\,A_3 &=N^2 - a_6\,C^2\label{obs4}, \\
    a_2\,a_3\,C_1^2 + a_1\,a_3\,C_2^2 + a_1\,a_2\,C_3^2 &= -2Mbv + C^2\,(3\,u^2+a_4\,v^2)\label{obs2},\\
    A_2\,A_3\,C_1^2 + A_1\,A_3\,C_2^2 + A_1\,A_2\,C_3^2 &= -2MN + a_4\,C^2\label{obs5},\\
    a_1\,C_2^2\,C_3^2 + a_2\,C_1^2\,C_3^2 + a_3\,C_1^2\,C_2^2 &= M^2\,v + 3C^2\,u\label{obs3}, \\
    A_1\,C_2^2\,C_3^2 + A_2\,C_1^2\,C_3^2 + A_3\,C_1^2\,C_2^2 &= M^2\label{obs6}.
\end{align}
Furthermore, we have the divisibility conditions
\begin{align}
    C_i\, C_j &\divides MN + A_i\,A_j\,C_k^2, \label{divis1}\\
    C_i \,C_j &\divides A_i\,A_j\,M + A_i\,C_j^2\,N + A_j\,C_i^2\,N \label{divis2}.
\end{align}
\end{lemma}
\begin{proof}
Because $P_1 + P_2 + P_3 = \mathcal{O}$ by hypothesis, we have \begin{equation*}
    \left(x - \frac{A_1}{C_1^2}\right)\left(x - \frac{A_2}{C_2^2}\right)\left(x - \frac{A_3}{C_3^2}\right) = x^3 + a_4x + a_6 - \left(\frac{m}{l}\,x + \frac{n}{l}\right)^2,
\end{equation*} or equivalently, \begin{equation*}
    \left(x - \frac{a_1}{C_1^2v}\right)\left(x - \frac{a_2}{C_2^2v}\right)\left(x - \frac{a_3}{C_3^2v}\right) = \left(x-\frac{u}{v}\right)^3 + a_4\left(x-\frac{u}{v}\right) + a_6 - \left(\frac{m}{l}\left(x-\frac{u}{v}\right)+\frac{n}{l}\right)^2.
\end{equation*} Expanding these equations, cross-multiplying, and comparing coefficients yields the first six identities. To obtain the divisibility condition~\eqref{divis1}, we will show that $C_i^2 C_j^2 \divides (MN + A_iA_jC_k^2)^2$. Using identities~\eqref{obs4}, \eqref{obs5}, and \eqref{obs6} to expand the product $(MN + A_iA_jC_k^2)^2$, we find that all terms vanish modulo $C_i^2C_j^2$. Similarly, condition~\eqref{divis2} follows by squaring and applying identities~\eqref{obs4}, \eqref{obs5}, and \eqref{obs6}.
\end{proof}

\begin{proof}[Proof of Lemma \ref{TheoremBhargava}]
First, we note that $\rho = v\cdot\frac{C}{g_1\,g_2\,g_3} \in \Z$ since $g_i\mid C_i$ for \(i \in \{1, 2, 3\}\). Second, we prove that $\psi_i = v\cdot\frac{q_i}{g_j\,g_k} \in \Z$ for each $i$. By the expressions for $q_i$ and $b$, we have 
\[
    \psi_i = \frac{v^3}{g_i\,g_j\,g_k}\cdot N\ell_i - \frac{u\,v^2\,\ell_i+v\,C_i\,\mu_i}{g_i\,g_j\,g_k}\cdot M.
\]

\noindent Since $g_i \mid v$, we have that $\frac{v^3}{g_i\,g_j\,g_k} \cdot N\ell_i\in\Z$. Furthermore, by identity \eqref{obs6} we have
\begin{align*}
\left(Mv\cdot\frac{u\,v\,\ell_i+C_i\,\mu_i}{g_i\,g_j\,g_k}\right)^2 = \left(A_i\,\frac{C_j^2\,C_k^2\,v^2}{g_i^2\,g_j^2\,g_k^2}+A_j\,\frac{C_i^2\,C_k^2\,v^2}{g_i^2\,g_j^2\,g_k^2}+A_k\,\frac{C_i^2\,C_j^2\,v^2}{g_i^2\,g_j^2\,g_k^2}\right)(u\,v\,\ell_i+C_i\,\mu_i)^2,
\end{align*}
which is clearly an integer.

Third, we prove that $\phi_i = \frac{v\,q_j\,q_k-a_i\,g_j\,g_k}{C\,g_i} \in \Z$. By the definitions of $q_i, \mu_i,$ and $\ell_i,$ we have
\[
    (C \,g_i\, g_j\, g_k)\, \phi_i = v\,(b\,v\,\ell_j-MC_j\,\mu_j)(b\,v\,\ell_k-MC_k\,\mu_k)-a_i\,(C_j^3\,\mu_j+a_j\,v\,\ell_j)(C_k^3\,\mu_k+a_k\,v\,\ell_k).
\]
We expand this expression as a linear combination of $\ell_j\, \ell_k, \, \mu_j \,\mu_k , \, \ell_j \,\mu_k, \, \ell_k \,\mu_j$, and we prove that each term is divisible by $Cg_i\,g_j\,g_k$. The first term in the expansion is given by $v\,\ell_j\,\ell_k\,(b^2\,v^2-a_i\,a_j\,a_k\,v)$. Note that by \eqref{obs1}, we have that
\[
v\,\ell_j\,\ell_k(b^2\,v^2-a_i\,a_j\,a_k\,v) = v\,\ell_j\,\ell_k\,\cdot (-C^2\,d),
\]
and since $Cg_i\,g_j\,g_k\mid C^2$, this term is divisible by $C\,g_i\,g_j\,g_k$.
The next term in the expansion is given by $C_j\,C_k\,\mu_j\,\mu_k\,(M^2\,v-a_i\,C_j^2\,C_k^2)$. By \eqref{obs3}, we see that
\[
C_j\,C_k\,\mu_j\,\mu_k\,(M^2\,v-a_i\,C_j^2\,C_k^2)=C_j\,C_k\,\mu_j\,\mu_k\,(a_j\,C_i^2\,C_k^2+a_k\,C_i^2\,C_j^2-3C^2\,u),
\]
which is divisible by $C\,g_i\,g_j\,g_k$.
Next, consider the term $-(bvM+a_i\,a_j\,C_k^2)\,v\,\ell_j\,C_k\,\mu_k$. Note that
\[
  -(bvM+a_i\,a_j\,C_k^2)v\,\ell_j\,C_k\,\mu_k = -((Nv-Mu)\,vM+(A_i\,v+C_i^2\,u)(A_j\,v+C_j^2\,u)C_k^2)v\,\ell_j\,C_k\,\mu_k.  
\]
Using identity \eqref{obs6}, we see that the right hand side of the above equation is equivalent modulo $C\,g_i \,g_j \,g_k$ to $-(MN+A_i\,A_j\,C_k^2)v^3\,\ell_j\,C_k\,\mu_k.$ Since $g_i\,g_j\,g_k\mid v^3$ and $C_i\,C_j \mid MN+A_i\,A_j\,C_k^2$ by condition~\eqref{divis1}, we have that $(Cg_i\, g_j\, g_k) \phi_i \equiv 0 \pmod {Cg_i\, g_j\, g_k}$. The same argument can be applied to the remaining term, $-(bvM+a_i\,a_k\,C_j^2)v\,\ell_k\,C_j\,\mu_j$. Thus, we conclude that $\phi_i \in \Z$.

Finally, we use a similar strategy to show that $\theta$ is an integer. In particular, we show that $(C^2\, g_1\, g_2 \,g_3) \theta$ is divisible by $C^2 \,g_1 \,g_2\, g_3.$ We again make the substitutions $g_i \,q_i = bv\ell_i - mC_i\, \mu_i$ and $g_i^2 = C_i^3 \,\mu_i + a_i\, v \ell_i$, and we see that $C^2 \,g_1 \,g_2 \,g_3 \,\theta$ is a linear combination of $\mu_i \,\mu_j \,\mu_k,$ $\ell_i \,\ell_j \,\mu_k$, $\ell_i\, \ell_j \,\ell_k,$ $\ell_i \,\mu_j\, \mu_k$ and their symmetric counterparts. We show that the coefficients of all such terms are divisible by $C^2\, g_1\, g_2 \,g_3.$ First, we consider the $\mu_1\, \mu_2 \,\mu_3$ term, which by identity~\eqref{obs3} has coefficient
    \begin{align*}
        2b C^3 + CM\,(a_1\, C_2^2\, C_3^2 + a_2\, C_1^2 \,C_3^2 + a_3\, C_1^2\, C_3^2)-CvM^3 = 2bC^3+3C^3\, uM.
    \end{align*}
Since $C^2\, g_1 \,g_2 \,g_3 \mid C^3,$ this coefficient is divisible by $C^2\, g_1 g_2 \,g_3$. Next, the $\ell_i \,\ell_j \,\mu_k$ term has coefficient $MC_k\, v (-b^2 v^2 + va_i \,a_j\, a_k) = MC_k\, v(C^2\, d),$ which is also divisible by $C^2\, g_1 \,g_2 \,g_3$. The coefficient of $\ell_1\, \ell_2\, \ell_3$ can be addressed similarly. Finally, for the $\ell_i\, \mu_j\, \mu_k$ coefficient, we expand the coefficient using the definitions of $b$ and $a_i.$ It is easy to see that $C^2 \,g_1\, g_2\, g_3$ divides the $u^2 v$ and the $uv^2$ terms using identity~\eqref{obs6}. We now address the $v^3$ term using \eqref{obs5} and \eqref{obs6}:
    \begin{align*}
        v^3 \,&C_2\, C_3 (A_1\,C_2^2\, C_3^2N +  A_1\, A_2\,C_3^2 M + C_2^2\, A_1 \,A_3 M + M^2N) \\
        &= v^2 \,C_2 \,C_3 \left((M^2\, N - A_2\, C_1^2\, C_3^2N - A_3 \,C_1^2\, C_2^2N) + M(-2MN + a_4\,C^2  -A_2 \,A_3\,C_1^2)+M^2 N\right) \\
        &= v^2 \,C_2 \,C_3\left( -C_1^2( A_2\, C_3^2\,N + A_3\, C_2^2\,N +  A_2\, A_3M) + a_4\, C^2  M\right).
    \end{align*}
Because $A_i\, C_j^2\,N + A_j \,C_i^2\,N + A_i\, A_j\, M$ is divisible by $C_i\, C_j$ by identity~\eqref{divis2}, we see that this term is indeed divisible by $C^2\, g_1\, g_2 \,g_3.$ This concludes the proof that $\theta$ is integral.
\end{proof}

Using Lemma \ref{TheoremBhargava}, we are now in a position to prove that $\Phi_{u,v}$ is a homomorphism.

\begin{proof}[Proof of Theorem~\ref{Theorem : Momomorphism}]
Given any three finite points $P_i \in E(\Q)$ that add to $\mathcal{O}$, we have exhibited an integral, primitive Bhargava cube associated to the forms $\Phi_{u,v}(P_i)$, implying that $\Phi_{u,v}(P_1) + \Phi_{u,v}(P_2) + \Phi_{u,v}(P_3)$ is the identity form. It is easily seen from the definition of $\Phi_{u, v}$ that it respects inverses, and so $\Phi_{u,v}$ is a homomorphism whenever $(u,v)$ is map-suitable for $E$: indeed, 
for any $P_1, P_2 \in E(\mathbb{Q})$ such that $P_1, P_2, P_1+P_2 \neq \mathcal{O}$, the Bhargava cube shows that $\Phi_{u,v}(P_1)+\Phi_{u,v}(P_2) + \Phi_{u,v}(-(P_1+P_2)) = 0$, so that $\Phi_{u,v}(P_1+P_2) = \Phi_{u,v}(P_1) + \Phi_{u,v}(P_2)$. 
\end{proof}

We conclude this section by giving conditions on $u$ and $v$ under which $\Phi_{u,v}$ is injective on the torsion subgroup, which allow us to identify subgroups of $\CL(-D_E(u,v))$ isomorphic to $E_{\mathrm{tor}}(\Q)$.

\begin{definition}\label{kernel-suitable}
Let $A_0 := \max\{A : (A, B) \in E_{\textrm{tor}}(\Q)\}$; we call a pair $(u,v)\in\Z^2$ \textbf{kernel-suitable for $\textbf{E}$} if $v > 1$ and $\frac{d_E(u, v) - uv}{v^2} > A_0$.
\end{definition}

\begin{corollary}\label{injective}
Let $(u,v)$ be kernel-suitable for $E$. Then $\Phi_{u,v}$ restricted to $E_{\mathrm{tor}}(\Q)$ is injective. 
\end{corollary}

\begin{proof}
If $P = (A, B) \in E_{\mathrm{tor}}(\Q)$ is in the kernel of $\Phi_{u,v}$, then by Theorem~\ref{TheoremPhi}, we have that $Av+u \geq \frac{d_E(u, v)}{v}$, so $\frac{d_E(u, v) - uv}{v^2} \leq A$, a contradiction.
\end{proof}

\section{Lower bound for $h(-D_E(u,v))$}\label{sec3}
We now use the map $\Phi_{u,v}$ to obtain an effective lower bound for the class number $h(-D_E(u,v))$. We will use Propositions 3.2 and 3.3 from \cite{G-O}, which combine to give the following result.

\begin{proposition}[Griffin-Ono]\label{prop: linearlyindependentboundedheights}
Let $\mathcal{P}$ be a set of $r$ linearly independent points in $E(\Q)$, and let $G$ be a subgroup of $E_{\mathrm{tor}}(\Q)$. Then for $T > \frac{d(\mathcal{P})}{4}$, we have \begin{equation*}
    \#\{P \in E(\Q) : \hat{h}(P) \leq T\} \geq c_G(\mathcal{P})\cdot \big(T^{\frac{r}{2}} - r\sqrt{d(\mathcal{P})} \cdot T^{\frac{r-1}{2}}\big).
\end{equation*}
\end{proposition}

Recall the definition of $T_E(u,v, \varepsilon)$ from \eqref{T}. To prove our lower bound, we also require a third notion of suitability.

\begin{definition}\label{bound-suitable}
We say that a pair $(u, v) \in \Z^2$ is $\bm{\varepsilon}$-\textbf{bound-suitable for $\mathcal{P}$} if $u, v > 0$ and
\begin{equation*}
\Big(\frac{1}{4}(uv+v^2)\Big)^{1+\varepsilon'} e^{2(1+ \varepsilon')(\delta(E)+d(\mathcal{P}))} < d_E(u,v) < (v^2(uv+v^2))^{1+\varepsilon'},
\end{equation*}
where $\varepsilon':= \frac{\varepsilon}{1-\varepsilon}$. 

\end{definition}

\noindent
Note that the first inequality in the above definition is equivalent to $T_E(u, v, \varepsilon) > d(\mathcal{P})/4$.

By showing that distinct points in $E(\Q)$ with canonical height bounded by $T_E(u,v, \varepsilon)$ map to distinct elements of the class group under $\Phi_{u,v}$, we obtain the following lower bound for $h(-D_E(u,v))$.

\begin{theorem}\label{suitability}
    Assume the hypotheses from Proposition~\ref{prop: linearlyindependentboundedheights}. Fix $0 < \varepsilon < \frac{1}{2}$, and let $\varepsilon' := \frac{\varepsilon}{1-\varepsilon}$. If $-D_E(u,v)$ is a negative fundamental discriminant for which $(u,v) \in \Z^2$ is map-suitable for $E$ and $\varepsilon$-bound-suitable for $\mathcal{P}$, then $$h(-D_E(u,v)) \geq c_G(\mathcal{P})\cdot \left( T_E(u,v, \varepsilon)^{\frac{r}{2}} - r \sqrt{d(\mathcal{P})} \cdot T_E(u,v, \varepsilon)^{\frac{r-1}{2}}\right).$$
\end{theorem}

\begin{proof}
Suppose that $(u,v)$ is map-suitable for $E$ and $\varepsilon$-bound-suitable for $\mathcal{P}$, and consider two points $P_1, P_2 \in E(\mathbb{Q})$ with canonical height bounded by $T_E(u,v, \varepsilon)$ and $\Phi_{u,v}(P_1) = \Phi_{u,v}(P_2)$. We show that $P := P_1 - P_2$ is the point of infinity.

Suppose that $P \neq \mathcal{O}$, so in particular $P = (\frac{A}{C^2}, \frac{B}{C^3})$ with $\gcd(A,C) = \gcd(B,C) = 1$ and $C >0$. Note that the triangle inequality implies $\hat{h}(P) \leq 4T_E(u,v, \varepsilon).$ By a theorem of Silverman \cite[Theorem~1.1]{Silverman},
\begin{equation*} 
    h_W(P) \leq 2\,(\hat{h}(P) + \delta(E) - \log 2) \leq 2\,(4T_E(u,v, \varepsilon) + \delta(E) - \log 2) \leq \log \bigg( \frac{d_E(u,v)^{1-\varepsilon}}{uv+v^2}\bigg).
\end{equation*}
In particular, this means that
\begin{equation}\label{IneqHeight}
    H(P)  = e^{h_W(P)}\leq \frac{d_E(u,v)^{1-\varepsilon}}{uv+v^2}.
\end{equation}

Since $(u,v)$ is map-suitable for $E$, we have $\Phi_{u,v}(P) = 0$, so
by Theorem \ref{TheoremPhi} either $\frac{va}{g^2} \geq d_E(u,v)$ or $v \mid C.$ However, by the definition of heights and equation \eqref{IneqHeight}, we have
    $$\frac{va}{g^2} \leq va = v|Av+C^2 u|\, \leq H(P)(uv+v^2) < d_E(u,v).$$
Thus, we only need to consider the case $v \mid C$, which implies that
    $$v \leq C \leq \max(|A|,C^2)^{\frac{1}{2}} = H(P)^{\frac{1}{2}}.$$
By \eqref{IneqHeight}, together with the above, we get $v^2 \leq \frac{d_E(u, v)^{1-\varepsilon}}{uv+v^2}$ so that we have
    $$d_E(u,v) \geq (v^2(uv+v^2))^{1+\varepsilon'}, \quad \text{where } \varepsilon' = \frac{\varepsilon}{1-\varepsilon}.$$

\noindent
However, since $(u,v)$ is $\varepsilon$-bound-suitable for $\mathcal{P}$, we have $d_E(u,v) < (v^2(uv+v^2))^{1+\varepsilon'}$. It follows that $ P = \mathcal{O}$, and hence that any two rational points $P_1 \neq P_2$ with height bounded by $T_E(u,v,\varepsilon)$ map to distinct forms.

From here, it suffices to count the number of points whose height is bounded by $T_E(u,v,\varepsilon)$. Since $(u,v)$ is \(\varepsilon\)-bound-suitable for $\mathcal{P}$, we have $T_E(u,v, \varepsilon) > \frac{d(\mathcal{P})}{4}$, so Proposition~\ref{prop: linearlyindependentboundedheights} implies that
\begin{equation*}
    h(-D) \geq \#\{P \in E(\Q): \hat{h}(P) \leq T_E(u,v, \varepsilon)\} \geq c_G(\mathcal{P})\cdot \big(T_E(u,v, \varepsilon)^{\frac{r}{2}} - r\sqrt{d(\mathcal{P})} \cdot T_E(u,v, \varepsilon)^{\frac{r-1}{2}}\big).
\end{equation*}
\end{proof}

\section{The Square-free Sieve}\label{sec4}

In this section, we will show that for any $\varepsilon_2 > 0$, the number of negative fundamental discriminants of the form $-X < -D_E(u,v) < 0$ that satisfy the conditions of Theorem \ref{BigTheorem} is asymptotically greater than $X^{\frac{1}{2}-\varepsilon_2}$. We study discriminants $-X<-D_E(x+ny,y)<0$, where $n \in \mathbb{Z}^+,$ such that $0<x,y<\lambda(n) X^{\frac14}$ for some constant $\lambda(n)>0$ that depends on the coefficients of the polynomial $-D_E(u,v)$. We use the following theorem, which is a special case of Theorem~3 and Proposition~5 of Gouv\^ea and Mazur \cite{GouveaMazur}.

\begin{theorem}[Gouv\^ea-Mazur]\label{Gouvea-Mazur}
Let $F(u,v) = v\cdot f(u,v)$, where $f(u,v)$ is a homogeneous polynomial of degree 3 with coefficients in $\Z$. Suppose that $F(u,v)$ has no square factors over $\Z[u,v]$, and that the greatest common divisor of all coefficients of $F(u,v)$ is 1. For integers $M, a_0, b_0$, suppose one of the two following conditions holds:
\begin{enumerate}
    \item We have $M = 2^{k}$ for $k \geq 2$, and $a_0,b_0$ are odd integers such that $F(a_0, b_0) \not\equiv 0 \pmod{4}$.
    \item The integers $a_0$ and $b_0$ are relatively prime to $M$ and \(F(a_0, b_0)\) is a unit modulo \(M\).
\end{enumerate}

\noindent
Finally, let $N(Y)$ denote the number of pairs of integers $(a,b)$ such that $0 \leq a, b \leq Y$, $a \equiv a_0 \pmod{M}$ and $b \equiv b_0 \pmod{M}$, and $F(a,b)$ is square-free.
Then there exists an \(A > 0\) such that
    $$N(Y) = A\,Y^2+O\left(\frac{Y^2}{\log(Y)^{\frac{1}{2}}}\right) \quad \text{as} \quad Y \rightarrow \infty.$$
\end{theorem}

Let $E: y^2=x^3+a_4x+a_6$ be an elliptic curve, and let $\mathcal{P}$ be a set of linearly independent points in $E(\Q)$. We recall our three suitability conditions for a pair $(u,v)\in\Z^2$. First, we say that $(u,v)$ is \textit{map-suitable  for $E$} if $u, v, D_E(u,v), 3u^2+a_4v^2 > 0$ (cf.\@ Theorem~\ref{TheoremPhi}). Second, we say that $(u, v)$ is \textit{kernel-suitable  for $E$} if $v > 1$ and $\frac{d_E(u,v)}{v} > A_0v+u$, where $A_0:= \max\{A : (A,B)\in E_{\textrm{tor}}(\Q)\}$ (cf.\@ Corollary~\ref{injective}). Third, we say that $(u, v)$ is \(\varepsilon\)-\textit{bound-suitable for $\mathcal{P}$} if  $$K(uv+v^2)^{1+\varepsilon'} < d_E(u,v) < (v^2(uv+v^2))^{1+\varepsilon'},$$ where \(K := \big(\frac 14\big)^{1+\varepsilon'}e^{2(1+ \varepsilon')(\delta(E)+d(\mathcal{P}))}\) (cf.\@ Theorem~\ref{suitability}). We show that for large enough $n$ asymptotically one hundred percent of pairs $(x+ny,y)$ with $0 < x,y < Y$ satisfy all of these suitability conditions as $Y \to \infty$.

\begin{lemma}\label{lemma: bound and kernel suitable}
Fix $0 < \varepsilon < \frac12$. Let $A_n(Y)$ denote the number of pairs of integers $0<x,y<Y$ such that \((x+ny, y)\) is not simultaneously map-suitable and kernel-suitable for $E$, and \(\varepsilon\)-bound-suitable for $\mathcal{P}$. Then there exists an \(N > 0\) such that for all \(n > N\), we have $A_n(Y) = O_{\varepsilon}(Y^{2-\varepsilon})$.
\end{lemma}

\begin{proof}
Let $G_n(x,y) := d_E(x+ny, y) = y(x^3 + 3nx^2y + (3n^2 + a_4)xy^2 + (n^3+a_4 n -a_6)y^3)$. There exists a constant $N > 0$ such that for $n> N$, each coefficient of $G_n(x,y)$ is positive. In particular, for $n > N$ and $x,y > 0$, we have $3(\frac{x+ ny}{y})^2 + a_4 > 3n^2 + a_4 > 0$ and $d_E(x+ny,y) > 0,$ so $(x+ny, y)$ is map-suitable for all $x, y > 0$. For brevity, we write
    $$G_n(x,y) =: y(x^3 + k_1(n) \,x^2 y + k_2(n)\, xy^2 + k_3(n)\, y^3).$$

The remaining suitability conditions are
\begin{enumerate}
    \item $G_n(x,y) > y(x+(n+A_0)\,y)$ and $y > 1$,
    \item $G_n(x,y) > K(y(x+(n+1)\,y))^{1+\varepsilon'}$, and
    \item $G_n(x,y) < (y^3(x+(n+1)\,y))^{1+\varepsilon'}.$
\end{enumerate}
First, consider integer pairs $(x,y)$ which do not satisfy property (1), i.e., $y = 1$ or
    $$x^3+k_1(n)\,x^2y+k_2(n)\,xy^2+k_3(n)\,y^3 \leq x +(n+A_0)\,y.$$
Dividing the inequality by $y^3$ yields
    $$\Big(\frac{x}{y}\Big)^3 + k_1(n) \cdot \Big(\frac{x}{y}\Big)^2 + \Big(k_2(n) - \frac{1}{y^2}\Big) \cdot \Big(\frac{x}{y}\Big) + \left(k_3(n) -\frac{n+A_0}{y^2}\right) \leq 0,$$
so the suitability condition fails only if $y \leq \max{\{1, k_2(n)^{-\frac{1}{2}}, \big(\frac{n+A_0}{k_3(n)}\big)^{\frac12}}\}$. Hence, the number of integer pairs $(x,y)$ which do not satisfy property (1) is $O(Y)$.

Second, suppose that $G_n(x,y) \leq K(y(x+(n+1)\,y))^{1+\varepsilon'}$. Because $x + (n+1)\,y \geq 1$ and $\varepsilon' \leq 1,$ we see that
    $$G_n(x,y) \leq K\,y^{1+\varepsilon'} (x + (n+1)\,y)^2.$$
Dividing both sides of the inequality by $y^4$ and simplifying, we get 
    $$\Big(\frac{x}{y}\Big)^3 + \Big(k_1(n) - \frac{K}{y^{1-\varepsilon'}}\Big)\cdot \Big(\frac{x}{y}\Big)^2 + \Big(k_2(n)-\frac{2K(n+1)}{y^{1-\varepsilon'}}\Big)\cdot \Big(\frac{x}{y}\Big) + \Big(k_3(n)-\frac{K(n+1)^2}{y^{1-\varepsilon'}}\Big) \leq 0.$$
Thus $y$ is similarly bounded and the number of pairs of integers $(x,y)$ which do not satisfy property (2) is $O(Y)$.

Finally, suppose that $G_n(x,y) \geq (y^3(x + (n+1)\,y))^{1+\varepsilon'}.$ Dividing through by $y^{4+4\varepsilon'}$ and setting $t:= \frac{x}{y},$ we see that
\begin{equation} \label{eq:SuitabilityUpperBound}
    \frac{1}{y^{4\varepsilon'}} \bigg(t^3 + k_1(n)\, t^2 + k_2(n)\, t + k_3(n)\bigg) -(t+n+1)^{1+\varepsilon'}\geq 0.
\end{equation}
First consider the case $t \leq 1,$ and let $C(n) := 1 + k_1(n) + k_2(n) + k_3(n).$ Then (\ref{eq:SuitabilityUpperBound}) implies that
    $$\frac{C(n)}{y^{4\varepsilon'}} - (n+1) \geq 0.$$
In particular, we have $y \leq \left(\frac{C(n)}{n+1}\right)^{\frac{1}{4\varepsilon'}},$ so the number of integer pairs $(x,y)$ with $x \leq y$ which do not satisfy property (3) is $O(Y).$ Now consider the case $t > 1$. Then (\ref{eq:SuitabilityUpperBound}) implies that
    $$C(n)\left(\frac{t}{y^{\varepsilon'}}\right)^3 - (n+1) \geq 0,$$
so that $y \leq \left( \frac{C(n)}{n+1} \right)^{\frac{1}{3 + 3\varepsilon'}} x^{\frac{1}{1+\varepsilon'}} < \left( \frac{C(n)}{n+1} \right)^{\frac{1-\varepsilon}{3}} Y^{1-\varepsilon}.$ Hence, the number of integer pairs $(x,y)$ with $x > y$ that do not satisfy property (3) is \(O_\varepsilon(Y^{2-\varepsilon})\).
\end{proof}

We now apply Theorem~\ref{Gouvea-Mazur} to count the number of negative fundamental discriminants satisfying the suitability conditions in Lemma \ref{lemma: bound and kernel suitable}. Let $E$ have conductor $N(E)$. Furthermore, for $0 < \varepsilon < \frac12$ and $\alpha \geq 0$, let $\mathcal{N}_Y(\varepsilon, \alpha, n)$ denote the set of pairs $(x, y) \in \Z^2$ with $0 < x,y < Y$ such that the following are true.
    \begin{enumerate}
        \item The pair $(x+ ny, y)$ is map-suitable and kernel-suitable for $E$.
        \item The pair $(x+ ny, y)$ is \(\varepsilon\)-bound-suitable and \((\varepsilon+\alpha)\)-bound-suitable for $\mathcal{P}$.
        \item The point $(-\frac{x}{y}-n, \frac{1}{y^2}) \in E_{-D_E(x+ny, y)}(\Q)$ has infinite order.
        \item The value $-D_E(x+ny, y)$ is a negative fundamental discriminant.
    \end{enumerate} 
Finally, for $\epsilon \in \{\pm 1\}$, let $\mathcal{M}_Y(\epsilon, \varepsilon, \alpha, n)$ denote the set of elements $(x, y)$ of $\mathcal{N}_Y(\varepsilon, \alpha, n)$ such that $d_E(x + ny, y)$ is coprime to $4N(E)$ and $\chi_{-d_E(x+ny, y)}(-N(E)) = \epsilon$, where $\chi_D$ is the quadratic Dirichlet character belonging to the field $\Q(\sqrt{D}).$
\begin{theorem} \label{Thm:CountN}
There exists a constant $A > 0$ and an $n \in \Z^+$ such that as $Y \to \infty$, we have \begin{equation*}
    \#\;\mathcal{N}_Y(\varepsilon,\alpha,n) = A\,Y^2 + O\left(\frac{Y^2}{\log(Y)^{1/2}}\right).
\end{equation*} Moreover, if $N(E)$ is not a perfect square and \(4N(E) \divides a_4, a_6\), then for \(\epsilon \in \{\pm 1\}\), there exists a constant $B > 0$ and an $n \in \Z^+$ such that as $Y \to \infty$, we have \begin{equation*}
    \#\;\mathcal{M}_Y(\epsilon, \varepsilon,\alpha, n) = B\,Y^2 + O\left(\frac{Y^2}{\log(Y)^{1/2}}\right).
\end{equation*}
\end{theorem}

\begin{proof}
By applying Lemma~\ref{lemma: bound and kernel suitable} twice, we find that there exists an $N \in \Z^+$ such that for all $n > N$, there are $o(Y^2)$ pairs $(x,y)$ with $0< x,y < Y$ such that $(x+ny,y)$ is not simultaneously (i) map-suitable for $E$, (ii) kernel-suitable for $E$, and (iii) \(\varepsilon\)-bound-suitable and \((\varepsilon + \alpha)\)-bound-suitable for $\mathcal{P}$. Moreover, by Proposition 1 of \cite{GouveaMazur}, only finitely many $E_{-D_E(x+ny, y)}$ can have a torsion point of order greater than \(2\). Thus, to prove the first assertion, it suffices to show that the number of integer pairs $(x,y)$ with $0< x,y < Y$ such that $-D_E(x+ny,y)$ is a fundamental discriminant is $\gg Y^2.$

We show that we can choose $n > N$ so that the result for $\mathcal{N}_Y(\varepsilon, \alpha, n)$ follows by applying Theorem~\ref{Gouvea-Mazur} to $G_n(x,y)$ with $M = 4.$ Observe that $\frac{G_n(x,y)}{y}$ is a homogeneous cubic polynomial with coefficients in $\mathbb{Z}$, and since $u \mapsto x+ny, v \mapsto y$ is an $\SL_2(\Z)$-transformation, $G_n(x,y)$ has no square factors as a polynomial in \(x, y\) if and only if $d_E(u,v)$ has no square factors as a polynomial in $u,v$. Note that $d_E(u,v) = v(u^3 + a_4 uv^2 - a_6 v^3) = -v^4 \big((-\frac{u}{v})^3 + a_4 (-\frac{u}{v}) + a_6\big)$ has a square factor only if $f(x) = x^3 + a_4x + a_6$ has a double root, which cannot occur since $E\colon\; y^2 = x^3+a_4 x + a_6$ is non-singular. The greatest common divisor of the coefficients of $G_n(x,y)$ is~\(1\), and one easily checks that regardless of the values of $a_4$ and $a_6$, there is a choice of $a_0, b_0, n \pmod 4$ such that $a_0$ and $b_0$ are odd and $G_n(a_0, b_0) \equiv 1$ or $2 \pmod 4$. This choice of $a_0$ and $b_0$ ensures that $-4G_n(x,y) = -D_E(x+ny,y)$ is a fundamental discriminant, and hence applying Theorem~\ref{Gouvea-Mazur} proves the first assertion.

To prove the second assertion, observe that in this case, we may choose \(n \equiv 0 \pmod{4N(E)}\) so that \(d_E(x+ny, y) \equiv x^3y \pmod{4N(E)}\). Note that $-D_E(x+ny, y)$ will be a fundamental discriminant if we choose $a_0 \equiv b_0 \equiv 1 \pmod{4}.$ Consequently, if we can choose \(d_0\) coprime to \(4N(E)\) such that \(\chi_{-d_0}(-N(E)) = \epsilon\), then \(d_E(x+ny, y)\) is coprime to \(4N(E)\) with \(\chi_{-d_E(x+ny, y)}(-N(E)) = \epsilon\) for all \(x\) and \(y\) in the congruence classes \(x \equiv 1 \pmod{4N(E)}\), \(y \equiv d_0 \pmod{4N(E)}\). If \(N(E)\) is odd, then \(\chi_{-d_0}(-N(E))\) is uniquely determined by the value of \(d_0 \pmod{N(E)}\) and achieves both signs for \(N(E)\) not a square, so we can always choose \(d_0 \equiv 1 \pmod{4}\) such that \(\chi_{-d_0}(-N(E)) = \epsilon\). Otherwise, write \(N(E) = 2^kN_0\). The case where \(k\) is even reduces to the former case. The case where \(k\) is odd can be addressed by choosing \(d_0\) to be in the correct residue class modulo \(8\) such that \(\chi_{-d_0}(-N_0) = -1\) and $d_0 \equiv 1 \pmod{4}$, so unless \(N(E)\) is a perfect square, we can always apply Theorem~\ref{Gouvea-Mazur} with \(a_0 :\equiv 1\) and \(b_0 :\equiv d_0 \pmod{4N(E)}\).
\end{proof}

\begin{proof}[Proof of Theorem \ref{BigTheorem}.] We give an argument similar to one used by Gouv\^ea and Mazur in \cite{GouveaMazur}. By Theorem~\ref{Thm:CountN}, there exists an \(n \in \Z^+\) such that \(\# \,\mathcal{N}_{\lambda(n) X^{1/4}}(\varepsilon_1,\alpha,n) = BX^{\frac12} + o(X^{\frac12})\), where $\lambda(n)$ is chosen so that every $-D \in \{-D_E(x+ny,y) : 0 < x,y < \lambda(n) X^{1/4}\} =: \mathcal{D}_X(n)$ satisfies $-X < -D < 0.$ 
To show that for any $\varepsilon_2 > 0$, the number of discriminants represented by elements of \(\mathcal{N}_{\lambda(n) X^{1/4}}(\varepsilon_1,\alpha,n)\) is also asymptotically at least \(X^{\frac{1}{2} - \varepsilon_2}\), we consider the map \begin{align*}
\gamma: \mathcal{N}_{\lambda(n) X^{\frac14}}(\varepsilon_1,\alpha,n) &\rightarrow \mathcal{D}_X(n)\\
(u,v) &\mapsto -D_E(u,v).
\end{align*}
We show that the cardinality of the fibers of $\gamma$ is $o(X^{\varepsilon_2})$. Let $-D$ be a negative fundamental discriminant and \((u,v) \in \gamma^{-1}(-D)\), and observe that $D = 4v(u^3 + a_4 uv^2 - a_6v^3)$ implies \(v \divides D\). In addition, note that for each fixed $v,$ since \(D_E(u,v)\) is cubic in \(u\), there are at most three choices for $u$ that give \(D_E(u,v) = D\). Because the number of positive divisors $v$ of $D$ is $o(X^{\varepsilon_2})$ for any $\varepsilon_2 > 0,$ we see that the cardinality of the fibers of $\gamma$ is $o(X^{\varepsilon_2}).$ For each \((x, y) \in \mathcal{N}_{\lambda(n) X^{1/4}}(\varepsilon_1,\alpha,n)\), the pair $(x+ny, y)$ is map-suitable and kernel-suitable for $E$ as well as $\varepsilon_1$-bound-suitable and $(\varepsilon_1+\alpha)$-bound-suitable for $\mathcal{P}$. In particular, $\varepsilon_1$-bound-suitability ensures that the lower bound for the class number given in Theorem~\ref{BigTheorem} is satisfied, and $(\varepsilon_1+\alpha)$-bound-suitability ensures that $T_E(u, v, \varepsilon_1) = \frac{\alpha}{8}\log(d_E(u, v)) + T_E(u, v, \varepsilon_1 + \alpha) > \frac{\alpha}{8}\log(d_E(u, v)) + \frac{d(\mathcal{P})}{4}$. Consequently,
we see that every discriminant in the image of \(\mathcal{N}_{\lambda(n) X^{1/4}}(\varepsilon_1,\alpha,n)\) satisfies the necessary conditions.



To prove the second part of the theorem, let $\epsilon$ denote the sign of the functional equation for the \(L\)-function associated to $E$, and consider the map \begin{align*}
\gamma': \mathcal{M}_{\lambda(n) X^{\frac14}}(\epsilon, \varepsilon_1,\alpha,n) &\rightarrow \mathcal{D}_X(n)\\
(u,v) &\mapsto -D_E(u,v).
\end{align*} Then for each \(-D \in \operatorname{Im}(\gamma\,')\), the sign of the functional equation for the \(L\)-function associated to $E_{-D}$ is equal to $\chi_{-\frac{D}{4}}(-N(E)) \cdot \epsilon = \epsilon^2 = 1$. Since \(E_{-D}\) has a point of infinite order, assuming the Parity Conjecture, we may conclude that $E_{-D}$ has even rank at least~$2$.
\end{proof}

\newpage \noindent
\appendix
\section{Discussion of Theorem ~\ref{TheoremInfiniteFamilies} and Explicit Formulas for the Infinite Families of Elliptic Curves}\label{Appendix}

In this section, we outline how to compute $$c_{\text{min}}(\mathcal{P}) := \frac{|G|}{\sqrt{{R}_{\Q}(\mathcal{P})}}\cdot \Omega_{r_\min(\mathcal{E}_G)}$$ for each of the torsion groups listed in Table~\ref{tabel: torsion subgroup}, where $r_{\mathrm{min}}(\mathcal{E}_G) > 0$ is a lower bound for the rank of a given infinite family of elliptic curves. For simplicity, we consider the case in which the family of interest is parametrized by only one variable. 

For a given infinite family of elliptic curves, we begin by computing\footnote{All computations were performed using \texttt{Sagemath} \cite{Sagemath}.} the short integral Weierstrass form $E^t: y^{2}  = x^{3} + a_4(t)\,x + a_6(t)$, along with the $x$-coordinates $\frac{r_i(t)}{s_i(t)}$ of the linearly independent points $P_i$ of infinite order on $E^t$. Since the N\'eron-Tate height pairing is an inner product, the matrix $(\langle P_i,P_j\rangle)_{1\leq i,j \leq r}$ is positive definite and symmetric. Hence, by Hadamard's inequality, the product of the diagonal entries $\langle P_i,P_i \rangle=\hat{h}(P_i)$ is an upper bound for $R_{\Q}(\mathcal{P})$. Thus, we have
\begin{equation}\label{eq: product}
   R_{\Q}(\mathcal{P}) = \det(\langle P_i,P_j\rangle)_{1\leq i,j \leq r} \leq \prod_i \langle P_i,P_i\rangle = \prod_i \hat{h}(P_i).
\end{equation}
To bound this product, we give an upper bound for each na\"ive height $H(P_i)$. Denote the coefficient of $t^k$ in the polynomial $q(t)$ by $q[t^k]$, and set $$p_i(t) := c_i\,(|t|+m_i)^{n_i},$$ where $n_i := \max\{\text{deg}(r_i),\text{deg}(s_i)\} $, $c_i := \max\{|r_i[t^{n_i}]|, |s_i[t^{n_i}]|\}$, and $m_i := \min\{b \in \Z^+: c_i \binom{n_i}{n}\,b^{n_i-n} \geq \max\{|r_i[t^{n}]|, |s_i[t^{n}]|\}, \text{ for all } 1\leq n \leq n_i \}$. Thus we obtain the upper bound $H(P_i) \leq p_i(t).$ Likewise, we obtain the bounds $$H(j(E^t)) \leq c_j\,(|t|+m_j)^{n_j} \qquad \text{and} \qquad H(\Delta(E^t)) \leq c_\Delta\,(|t| + m_{\Delta})^{n_{\Delta}}.$$
We now use these upper bounds to derive an upper bound for $\hat{h}(P_i)$ using a theorem of Silverman \cite[Theorem~1.1]{Silverman} that relates the Weil heights to the canonical heights. 

\begin{theorem}[Silverman]\label{thm: silverman}
If $P \in E(\Q)$, then $$-\frac18 h_W(j(E))-\frac{1}{12}h_W(\Delta(E))-0.973 \leq \hat{h}(P)-\frac12 h_W(P) \leq \frac{1}{12}h_W(j(E))+\frac{1}{12}h_W(\Delta(E))+1.07.$$
\end{theorem}

\noindent
Let $m := \max(\{m_i\}_{i\,=\,1}^{r} \cup \{m_j, m_{\Delta} \})$. Then we have
\begin{equation*}
    \hat{h}(P) \leq \left(\frac{1}{12}\log(c_j) + \frac{1}{12}\log(c_\Delta) + \frac{1}{2}\log(c_i) + 1.07\right) + \left(\frac{1}{12}n_j + \frac{1}{12}n_\Delta + \frac{1}{2}n_i\right)\log(|t|+m).
\end{equation*}
Hence, we obtain an upper bound for $R_{\Q}(\mathcal{P})$. Finally, we choose $\mu$ so that $(\log(|t|+m)+\mu)^r$ is an upper bound for the product in \eqref{eq: product}. We then take the square root of this monomial to obtain $c_{\text{min}}(\mathcal{P})$, as listed in Table~\ref{tabel: torsion subgroup}.

To see explicit formulas for each of the infinite families and the $x$-coefficients of their linearly independent points of infinite order, as well as the code created to find these polynomials and the values listed in Table~\ref{tabel: torsion subgroup}, please visit \url{https://github.com/team-class-numbers/elliptic-curve-families}.

\end{document}